\documentclass[12pt]{amsart}
\usepackage{fullpage,url,amssymb,enumerate,colonequals}
\usepackage[all,pdf]{xy} 

\usepackage{comment}
\usepackage{color}
\usepackage{charter,euler} 
\DeclareSymbolFont{bchoperators}{T1}{bch}{m}{n}
\SetSymbolFont{bchoperators}{bold}{T1}{bch}{b}{n}
\makeatletter
\renewcommand{\operator@font}{\mathgroup\symbchoperators}
\makeatother

\DeclareFontEncoding{OT2}{}{} 
\newcommand{\textcyr}[1]{%
 {\fontencoding{OT2}\fontfamily{wncyr}\fontseries{m}\fontshape{n}\selectfont #1}}
\newcommand{\Sha}{{\mbox{\textcyr{Sh}}}}

\usepackage{titlesec} 
\titleformat{\section}{\normalfont\bfseries\filcenter}{\thesection}{1em}{}

\usepackage[pdftex]{epsfig}

\usepackage{todonotes}

\newcommand{\C}{\mathbb{C}}
\newcommand{\F}{\mathbb{F}}

\newcommand{\PP}{\mathbb{P}}
\newcommand{\Q}{\mathbb{Q}}
\newcommand{\R}{\mathbb{R}}

\newcommand{\Z}{\mathbb{Z}}

\newcommand{\Sp}{\operatorname{Sp}}

\newcommand{\tors}{\operatorname{tors}}

\newcommand{\im}{\operatorname{im}}
\newcommand{\rk}{\operatorname{rk}}
\newcommand{\Spec}{\operatorname{Spec}}
\newcommand{\Sel}{\operatorname{Sel}}
\newcommand{\known}{\operatorname{known}}

\newcommand{\To}{\longrightarrow}

\numberwithin{equation}{section}

\newtheorem{theorem}{Theorem}
\newtheorem{lemma}[theorem]{Lemma}
\newtheorem{corollary}[theorem]{Corollary}
\newtheorem{proposition}[theorem]{Proposition}

\theoremstyle{definition}

\newtheorem{conjecture}[theorem]{Conjecture}

\theoremstyle{remark}

\newtheorem*{remark*}{Remark}

\definecolor{darkgreen}{rgb}{0,0.5,0}
\definecolor{rem}{rgb}{0.8,0,0}
\definecolor{new}{rgb}{0.7,0,0.6}
\definecolor{reply}{rgb}{0,0,0.8}
\usepackage[
        colorlinks, citecolor=darkgreen,
        backref=page,
        pdfauthor={Hang Fu, Michael Stoll},
        pdftitle={Dynamics of quadratic polynomials and
        rational points on a curve of genus 4},
]{hyperref}
\usepackage[alphabetic,backrefs,lite]{amsrefs} 

\allowdisplaybreaks

\begin{document}

\title{Dynamics of quadratic polynomials and \\ rational points on a curve of genus $4$}

\author{Hang Fu}
\address{Department of Mathematics,
        National Taiwan University,
        Taipei, Taiwan}
\email{drfuhang@gmail.com}
\urladdr{https://sites.google.com/view/hangfu}

\author{Michael Stoll}
\address{Mathematisches Institut,
         Universit\"at Bayreuth,
         95440 Bayreuth, Germany.}
\email{Michael.Stoll@uni-bayreuth.de}
\urladdr{http://www.mathe2.uni-bayreuth.de/stoll/}

\date{\today}

\begin{abstract}
  Let $f_t(z)=z^2+t$. For any $z\in\Q$, let $S_z$ be the collection of $t\in\Q$ such that $z$ is preperiodic for $f_t$. In this article, assuming a well-known conjecture of Flynn, Poonen, and Schaefer, we prove a uniform result regarding the size of $S_z$ over $z\in\Q$. In order to prove it, we need to determine the set of rational points on a specific non-hyperelliptic curve $C$ of genus $4$ defined over $\Q$. We use Chabauty's method, which requires us to determine the Mordell-Weil rank of the Jacobian $J$ of $C$. We give two proofs that the rank is $1$: an analytic proof, which is conditional on the BSD rank conjecture for $J$ and some standard conjectures on L-series, and an algebraic proof, which is unconditional, but relies on the computation of the class groups of two number fields of degree $12$ and degree $24$, respectively. We finally combine the information obtained from both proofs to provide a numerical verification of the strong BSD conjecture for $J$.
\end{abstract}

\subjclass{11G30, 11G40, 14G05, 14G10, 14H45, 37P05}

\keywords{Preperiodic points, Rational points, BSD conjecture}

\maketitle

\section{Introduction}

Let $K$ be a number field of degree $D$ and let $f \colon \PP^n\to\PP^n$ be a morphism of degree $d\ge2$
and defined over~$K$. A point $P\in \PP^n(K)$ is said to \emph{preperiodic} for $f$ if its forward orbit $\{f^n(P):n\ge0\}$ is finite. Let
\[
\text{PrePer}(f,K)=\{P\in \PP^n(K):P\text{ is preperiodic for }f\}
\]
be the set of all preperiodic points of $f$ defined over $K$. The following Uniform Boundedness Conjecture was proposed by Morton and Silverman \cite{MS}.

\begin{conjecture}
There exists a uniform bound $B=B(D,n,d)$ such that $\#\text{PrePer}(f,K)\le B$ for any number field $K$ of degree $D$ and any morphism $f:\PP^n\to\PP^n$ of degree $d\ge2$ defined over~$K$.
\end{conjecture}

This conjecture seems to be very difficult and even the simplest case $(D,n,d)=(1,1,2)$ is largely open. Flynn, Poonen, and Schaefer \cite{FPS} conjectured that

\begin{conjecture}\label{conj:FPS}
If $N\ge4$, then there is no quadratic polynomial $f(z)\in\Q[z]$ with a rational point of exact order $N$.
\end{conjecture}

The case $N=4$ has been verified by Morton \cite{Morton98}, the case $N=5$ has been verified by Flynn, Poonen, and Schaefer \cite{FPS}, and the case $N=6$ has been verified by the second named author \cite{Stoll2008a} subject to the validity of the BSD conjecture and some standard conjectures on L-series. Poonen \cite{Poonen98}*{Cor. 1} showed that if Conjecture \ref{conj:FPS} is true, then $\#\text{PrePer}(f,\Q)\le 9$ for any quadratic polynomial $f(z)\in\Q[z]$.

Now we consider this problem  in an opposite way. Let $f_{t}(z)=z^{2}+t$. For any $z\in\Q$, let
\[
S_{z}=\{t\in\Q:z\text{ is preperiodic for }f_{t}\}.
\]
Note that $S_{z}=S_{-z}$ for any $z\in\Q$. In this article, we determine the size of $S_{z}$ for each $z\in\Q$, subject to the validity of Conjecture~\ref{conj:FPS}.

\begin{theorem}\label{T:SZ}
If Conjecture \ref{conj:FPS} is true, then
\begin{enumerate}[\upshape(1)]
\item $\#S_{z}=3$ if and only if $z=0$ or $\pm\tfrac12$.
\item $\#S_{z}=5$ if and only if
\[
z=\pm\frac{a^{3}-a-1}{2a(a+1)} \quad\text{for some $a\in\Q\backslash\{-1,-\tfrac12,0,2\}$.}
\]
\item $\#S_{z}=6$ if and only if
\[
z=\frac{2a}{a^{2}-1} \quad\text{for some $a\in\Q\backslash\{0,\pm\tfrac15,\pm\tfrac13,\pm1,\pm3,\pm5\}$,}
\]
or
\[
z=\frac{a^{2}+1}{a^{2}-1} \quad\text{for some $a\in\Q\backslash\{0,\pm\tfrac13,\pm1,\pm3\}$.}
\]
\item $\#S_{z}=7$ if and only if $z=\pm\tfrac5{12}$, $\pm\tfrac34$, or $\pm\tfrac54$.
\item $\#S_{z}=4$ otherwise.
\end{enumerate}
In particular, $\#S_{z}$ is uniformly bounded (by~$7$) over $z\in\Q$.
\end{theorem}

The proof of Theorem \ref{T:SZ} can be easily reduced to finding the rational points on five specific algebraic curves. Among these curves, two have genus $1$, two have genus $2$, and one has genus $4$. We will see that the conductors of the four lower genus curves are small enough so that their rational points can be found at \emph{the L-functions and modular forms database} \cite{LMFDB}.

The main difficulty is to find the rational points on the genus $4$ curve, an affine plane model of which is given by the equation
\begin{equation} \label{E:curve}
  x^{3}y^{2} + x^{2}y^{3} - x^{3}y - xy^{3} - x^{2}y - xy^{2} + x^{2} + 2xy + y^{2} - x - y = 0 \,.
\end{equation}

Let $C$ be the closure in $\PP^1\times\PP^1$ of this affine curve. It can be easily checked that $C$ is smooth. Since $C$ has genus $4$, the set $C(\Q)$ of rational points is finite by Faltings's Theorem \cite{Faltings}. We will prove:

\begin{theorem} \label{T:main}
\[ C(\Q) = \{(0,0), (0,1), (0,\infty), (1,0), (1,1), (1,\infty),
             (\infty,0), (\infty,1), (\infty,\infty)\} \,.
\]
\end{theorem}

In fact, we will give two proofs: one that is conditional on the Birch
and Swinnerton-Dyer rank conjecture for the Jacobian variety~$J$ of~$C$ (plus
standard conjectures on $L$-series) and one that is unconditional and
based on the determination of the (size of) the $2$-Selmer group of~$J$.
The first proof follows~\cite{Stoll2008a}, where the set of rational
points on a different curve of genus~$4$ was determined. The second proof
applies the general descent machinery developed in~\cite{BPS2016} to
our specific curve. In both cases, the crucial step is to bound the
rank of the Mordell-Weil group~$J(\Q)$ by~$1$. In the first, ``analytic'',
proof, this is done by verifying numerically that the derivative of
the $L$-series of~$J$ at $s = 1$ is nonzero, whereas in the second,
``algebraic'', proof, the bound comes from the Selmer group.

The latter approach was infeasible for the curve considered in~\cite{Stoll2008a},
since it requires the computation of information like the class and unit
groups of the number fields over which the various odd theta characteristics
are defined. This would involve a field of degree~$119$ in the case
of~\cite{Stoll2008a}. In contrast, all odd theta characteristics of the curve~$C$
studied here are defined over number fields of degree at most~$24$, for which the
necessary data can be computed unconditionally in reasonable time.

We then combine the computation of the leading term of the
$L$-series of~$J$ at~$s = 1$ from the analytic proof with the unconditional
algebraic proof and some further computations to find that the strong
version of the Birch and Swinnerton-Dyer conjecture predicts that
$\Sha(J/\Q)$ is trivial (i.e., the ``analytic order of $\Sha$''
is equal to~$1$ to the precision of the computation), which is consistent
with what we know and expect about~$\Sha(J/\Q)$.

The structure of this paper is as follows. In Section \ref{S:SZproof}, we first assume Theorem \ref{T:main}  to give the proof of Theorem \ref{T:SZ}. In Section \ref{S:prop}, we give some facts on the curve $C$ and its Jacobian $J$ that are needed for both proofs of Theorem \ref{T:main}. In Section \ref{S:analytic} and Section \ref{S:algebraic}, we give the analytic proof and the algebraic proof of Theorem \ref{T:main}. Finally, in Section \ref{S:BSD}, we provide the numerical verification of the strong BSD conjecture for $J$.

The necessary computations have been performed using Magma~\cite{Magma}.
The code that verifies the various computational assertions made
in this paper can be found at~\cite{Code}.

\subsection*{Acknowledgments.}

We thank Jan Steffen M\"{u}ller for providing us with code for the computation
of canonical heights on~$J$ and Jakob Stix for helping with finding a reference
for the relation between the canonical height of a point on~$J$ and that of
its image on the quotient elliptic curve.


\section{Proof of Theorem \ref{T:SZ}} \label{S:SZproof}

Following the terminology of \cite{Poonen98}, we say that $z\in\Q$ is \emph{of type~$m_{n}$} if there exists some $t\in\Q$
such that the orbit $\{f^k_t(z):k\ge0\}$ enters an
$m$-cycle after $n$ iterations.

If $z$ is of type $1_{0}$, $2_{0}$, $1_{1}$, or $2_{1}$, then $f_{t}^{3}(z)=f_{t}(z)$, which implies
\[
t\in T_{z}:=\{-z^{2}\pm z,-z^{2}\pm z-1\}\subseteq S_{z}\,.
\]
It is easy to see that $\#T_{z}=2$ for $z=0$, $\#T_{z}=3$ for $z=\pm\tfrac12$,
and $\#T_{z}=4$ otherwise. For other types, Poonen proved the following result.

\begin{theorem}[\cite{Poonen98}*{Thms.~1, 3}]\label{T:types}
Let $f_t(z)=z^2+t$ with $t\in\Q$. Then
\begin{enumerate}[\upshape(1)]
\item 
$z$ is of type $3_{0}$ if and only if
\[
z=g_{30}(a):=\frac{a^{3}-a-1}{2a(a+1)} \quad\text{for some $a\in\Q\backslash\{-1,0\}$.}
\]
In this case,
\[
t=h_{30}(a):=-\frac{a^6+2a^5+4a^4+8a^3+9a^2+4a+1}{4a^2(a+1)^2}\,.
\]

\item 
$z$ is of type $3_{1}$ for some $t$ if and only if $-z$ is of type $3_{0}$ for the same~$t$.

\item 
$z$ is of type $1_{2}$ if and only if
\[
z=g_{12}(a):=\frac{2a}{a^{2}-1} \quad\text{for some $a\in\Q\backslash\{\pm1\}$.}
\]
In this case,
\[
t=h_{12}(a):=-\frac{2(a^2+1)}{(a^2-1)^2}\,.
\]

\item 
$z$ is of type $2_{2}$ if and only if
\[
z=g_{22}(a):=\frac{a^{2}+1}{a^{2}-1} \quad\text{for some $a\in\Q\backslash\{0,\pm1\}$.}
\]
In this case,
\[
t=h_{22}(a):=-\frac{a^4+2a^3+2a^2-2a+1}{(a^2-1)^2}\,.
\]

\item 
$z$ is of type $3_{2}$ if and only if $z=\pm\tfrac34$.
In this case, $t=-\tfrac{29}{16}$.

\item 
$z$ cannot be of type $1_n$, $2_n$, or $3_n$ for any $n\ge3$.
\end{enumerate}
\end{theorem}

\begin{remark*}
We note that $h_{30}(a)=h_{30}(-1/(a+1))=h_{30}(-(a+1)/a)$, $h_{12}(a)=h_{12}(-a)$, and $h_{22}(a)=h_{22}(-1/a)$.
\end{remark*}

Now we assume Theorem \ref{T:main} to give the proof of Theorem \ref{T:SZ}.

\begin{proof}[Proof of Theorem \ref{T:SZ}]

By Conjecture \ref{conj:FPS} and Theorem \ref{T:types}(6), if $z$ is of type $m_n$, then $1\le m\le3$ and $0\le n\le2$. Now we need to find all coincidences
that $z$ is of type $m_{n}$ for $t$ and of type $m'_{n'}$ for $t'$. Note that if $z$
is of type $1_{2}$ or $2_{2}$, then $-z$ is also of type $1_{2}$ or $2_{2}$.

\begin{enumerate}[(1)]\addtolength{\itemsep}{3pt}
\item $z=0$ is of type $1_2$, but not of type $3_0$, $3_1$, $2_2$, or $3_2$.

\item $z=\pm\tfrac12$ are not of type $3_0$, $3_1$, $1_2$, $2_2$, or $3_2$.

\item $z=\pm\tfrac34$ are of types $1_{2}$ and $3_{2}$, but not of type $3_{0}$, $3_{1}$, or $2_{2}$.

\item If $z\neq0$ is of type $1_{2}$, then $z=g_{12}(a)=g_{12}(-1/a)$, which
gives two points in $S_{z}$.

\item If $z$ is of type $2_{2}$, then $z=g_{22}(a)=g_{22}(-a)$, which gives
two points in $S_{z}$.

\item If $z$ is of type $3_{0}$ for some $t\neq t'$, then $z=g_{30}(a)=g_{30}(b)$
for some $a\neq b$. Then we can see that $(a,b)$ is a rational point on the curve
\[
C_1 \colon a^{2}b^{2}+a^{2}b+ab^{2}+ab+a+b+1=0\,.
\]
This is an elliptic curve with LMFDB label  \href{https://www.lmfdb.org/EllipticCurve/Q/14/a/5}{14.a5}.
The set of rational points is
\[
C_1(\Q)=\{(-1,0),(0,-1),\text{ and four at infinity}\}\,.
\]
Therefore, this case cannot occur.

\item If $z$ is of types $1_{2}$ and $2_{2}$, then $z=g_{12}(a)=g_{22}(b)$
for some $a$ and $b$. Then we can see that $(a,b)$ is a rational point on the curve
\[
C_2 \colon a^{2}b^{2}-2ab^{2}+a^{2}-b^{2}+2a-1=0\,.
\]
This is an elliptic curve with LMFDB label  \href{https://www.lmfdb.org/EllipticCurve/Q/32/a/3}{32.a3}.
The set of rational points is
\[
C_2(\Q)=\{(\pm1,\pm1)\}\,.
\]
Therefore, this case cannot occur.

\item If $z$ is of types $3_{0}$ and $1_{2}$, then $z=g_{30}(a)=g_{12}(b)$
for some $a$ and $b$. Then we can see that $(a,b)$ is a rational point on the curve
\[
C_3 \colon a^{3}b^{2}-a^{3}-4a^{2}b-ab^{2}-4ab-b^{2}+a+1=0\,.
\]
This is a genus $2$ curve with LMFDB label  \href{https://www.lmfdb.org/Genus2Curve/Q/953/a/953/1}{953.a.953.1}.
The set of rational points is
\[
C_3(\Q)=\{(-1,\pm1),(0,\pm1),(2,-\tfrac15),(2,5),\text{ and two at infinity}\}\,.
\]
Therefore, $g_{30}(2)= \tfrac5{12}$ is of types $3_{0}$ and $1_{2}$, and $-\tfrac5{12}$ is of types $3_{1}$ and $1_{2}$.

\item If $z$ is of types $3_{0}$ and $2_{2}$, then $z=g_{30}(a)=g_{22}(b)$
for some $a$ and $b$. Then we can see that $(a,b)$ is a rational point on the curve
\[
C_4 \colon a^{3}b^{2}-2a^{2}b^{2}-a^{3}-3ab^{2}-2a^{2}-b^{2}-a+1=0\,.
\]
This is a genus $2$ curve with LMFDB label  \href{https://www.lmfdb.org/Genus2Curve/Q/713/a/713/1}{713.a.713.1}.
The set of rational points is
\[
C_4(\Q)=\{(-1,\pm1),(-\tfrac12,\pm3),(0,\pm1),\text{ and two at infinity}\}\,.
\]
Therefore, $g_{30}(-\tfrac12)=\tfrac54$ is of types $3_{0}$ and $2_{2}$, and $-\tfrac54$ is of types $3_{1}$ and $2_{2}$.

\item If $z$ is of types $3_{0}$ and $3_{1}$, then $z=g_{30}(a)=-g_{30}(b)$
for some $a$ and $b$. Then we can see that $(a,b)$ is a rational point on the curve
\[
C_5 \colon a^{3}b^{2}+a^{2}b^{3}+a^{3}b+ab^{3}-a^{2}b-ab^{2}-a^{2}-2ab-b^{2}-a-b=0\,.
\]
Let $(a,b)=(-x,-y)$; then by Theorem \ref{T:main}, we know that the set of rational points is
\[
C_5(\Q)=\{(-1,-1),(-1,0),(0,-1),(0,0),\text{ and five at infinity}\}\,.
\]
Therefore, this case cannot occur.
\end{enumerate}

Finally, let us give a brief summary. For most $z$, $\#S_z=\#T_z=4$. For most $z$ of type $3_0$ or $3_1$, $\#S_z=\#T_z+1=5$. For most $z$ of type $1_2$ or $2_2$, $\#S_z=\#T_z+2=6$. The exceptional cases are $\#S_0=\#S_{\pm1/2}=3$ and $\#S_{\pm5/12}=\#S_{\pm3/4}=\#S_{\pm5/4}=7$.
\end{proof}

\begin{remark*}
We find that $S_{\pm3/4}$ and $S_{\pm5/4}$ have a large intersection compared to their own sizes.
\begin{align*}
  S_{\pm3/4} & \supseteq\{\hphantom{-\tfrac{61}{16}, {}} {-\tfrac{45}{16}}, -\tfrac{37}{16}, -\tfrac{29}{16},
                          -\tfrac{21}{16}, -\tfrac{13}{16}, -\tfrac{5}{16}, \tfrac{3}{16}\}\,,\\
  S_{\pm5/4} & \supseteq\{-\tfrac{61}{16}, -\tfrac{45}{16}, -\tfrac{37}{16}, -\tfrac{29}{16},
                          -\tfrac{21}{16}, -\tfrac{13}{16}, -\tfrac{5}{16}\hphantom{, \tfrac{3}{16}}\}\,.
\end{align*}
If Conjecture \ref{conj:FPS} is true, then $\max\#(S_z \cap S_{z'})=6$, where $z\ne \pm z'$.
\end{remark*}

\section{Properties of the curve $C$ and its Jacobian $J$} \label{S:prop}

We now proceed with the proof of Theorem~\ref{T:main}.
We begin with some basic facts. Recall that we take $C$ to be
the closure in $\PP^1 \times \PP^1$ of the affine curve given
by~\eqref{E:curve}.

\begin{lemma} \label{L:1}
  The curve~$C$ is smooth of genus~$4$. It has good reduction
  at all primes \hbox{$p \notin \{2, 23, 29\}$}. For each of the bad
  primes~$p$, the given model of~$C$ is regular over~$\Z_p$.
  The special fibers at the bad primes are as follows.
  \begin{itemize}
    \item[$p = 2$:] Three smooth curves of genus~$0$ defined over~$\F_2$,
      all meeting in two points that are swapped by Frobenius.
    \item[$p = 23$:] One component (of geometric genus~$3$)
      with a single non-split node.
    \item[$p = 29$:] One component (of geometric genus~$2$)
      with two split nodes that are swapped by Frobenius.
  \end{itemize}
\end{lemma}

\begin{proof}
  We consider each of the four affine patches (actually, because
  of the symmetry of the equation, three are sufficient). By a
  Gr\"{o}bner basis computation over~$\Z$, we determine the support
  of the projection of the singular subscheme of~$C$ to~$\Spec \Z$.
  In each case, we find that the support is contained in the set
  $\{2, 23, 29\}$. For each of the bad primes~$p$, we determine the
  geometry of the special fiber by reducing the equation defining~$C$
  modulo~$p$. We then check that the given model is regular by
  looking at the curve locally near each of the singularities
  on the special fibers.
\end{proof}

\begin{lemma} \label{L:2}
  The curve~$C$ has an involution~$\tau$ given by $x \leftrightarrow y$.
  The quotient $E = C/\langle \tau \rangle$ is isomorphic to the
  elliptic curve with LMFDB
  label~\href{https://www.lmfdb.org/EllipticCurve/Q/58/a/1}{58.a1}.
  The Mordell-Weil group of~$E$ is isomorphic to~$\Z$.
  The quotient abelian variety $A = J/E$ is geometrically simple.
\end{lemma}

\begin{proof}
  The existence of~$\tau$ is obvious from~\eqref{E:curve}.
  The fixed points of~$\tau$ are given by setting $y \leftarrow x$
  in the (homogenized) equation; this leads to a sextic binary form
  that is squarefree, so $\tau$ has six simple fixed points.
  By Riemann-Hurwitz, we conclude that the quotient~$E$ has genus~$1$.
  Since $C$ has rational points, the same holds for~$E$, and so
  $E$ is an elliptic curve. Using Magma, we can construct the quotient
  and identify the curve. The information on~$E(\Q)$ is available
  from the LMFDB~\cite{LMFDB}. Finally, we show that $A$ is geometrically
  simple using the argument employed for~\cite{Stoll2008a}*{Lemma~2}:
  We compute the characteristic polynomials of Frobenius for~$A$
  at the primes $3$ and~$7$ as quotients of the corresponding polynomials
  for~$J$ (which we get from the zeta function of~$C$) and~$E$.
  We then check that they are irreducible, satisfy the criterion
  of~\cite{Stoll2008a}*{Lemma~3} and define linearly disjoint
  number fields.
\end{proof}

\begin{lemma} \label{L:3}
  The Jacobian~$J$ has conductor $2^6 \cdot 23 \cdot 29^2 = 1\,237\,952$,
  and the cofactor $A = J/E$ has conductor~$2^5 \cdot 23 \cdot 29 = 21\,344$.
\end{lemma}

\begin{proof}
  The second statement follows from the first, since $E$ has conductor
  $58 = 2 \cdot 29$ by Lemma~\ref{L:2}. To prove the first statement,
  consider the special fibers described in Lemma~\ref{L:1}.
  We denote the dimensions of the abelian, toric, and unipotent parts
  of the special fiber of the N\'eron model at the prime~$p$
  by $a_p$, $t_p$, and~$u_p$, respectively. Recall that the exponent
  of the conductor at~$p$ is at least $t_p + 2u_p$, with equality
  when $u_p = 0$ (or when $p$ is sufficiently large). The difference
  is the ``wild part'' of the conductor exponent.
  \begin{itemize}
    \item[$p = 29$:] $a_{29} = 2$, $t_{29} = 2$, $u_{29} = 0$,
      so the exponent of the conductor at~$29$ is~$2$.
    \item[$p = 23$:] $a_{23} = 3$, $t_{23} = 1$, $u_{23} = 0$,
      so the exponent of the conductor at~$23$ is~$1$.
    \item[$p = 2$:] $a_2 = 0$, $t_2 = 2$, $u_2 = 2$,
      so the exponent of the conductor at~$2$ is at least~$6$.
      We show by a computation that over the tamely ramified extension
      $\Q_2(\sqrt[3]{2})$ of~$\Q_2$, $C$ acquires semistable reduction,
      with special fiber consisting of two curves of genus~$1$ meeting
      in three points. This implies that there is no wild part in
      the conductor, and hence the exponent is indeed~$6$.
      (This is similar to the argument used in the proof
      of~\cite{FHS2011}*{Prop.~7.2}.)
    \qedhere
  \end{itemize}
\end{proof}

\begin{lemma} \label{L:4}
  The statement of Theorem~\ref{T:main} follows when $J(\Q)$ has rank~$1$.
\end{lemma}

\begin{proof}
  We use Chabauty's method; see for example~\cite{Stoll2006a}.
  Since we know that $E(\Q)$ has rank~$1$ (Lemma~\ref{L:2}), the assumption
  on~$J(\Q)$ implies that $A(\Q)$ is torsion; therefore, for any prime~$p$,
  the subspace of the space of $p$-adic regular differentials on~$C$
  consisting of those differentials that annihilate the
  Mordell-Weil group under the Chabauty-Coleman pairing is exactly the
  $(-1)$-eigenspace under~$\tau$ (this corresponds to the space of invariant
  differential $1$-forms on~$A$). These differentials correspond to the $(1,1)$-forms
  on~$\PP^1 \times \PP^1$ that are fixed by the automorphism given by swapping
  the two factors. For every prime~$p$ and every point in $(\PP^1 \times \PP^1)(\F_p)$,
  we can find a symmetric $(1,1)$-form defined over~$\F_p$ that does not
  vanish on the given point. If $p \ge 3$ is a prime of good reduction
  (i.e., $p \neq 23, 29$), then \cite{Stoll2006a}*{Lemma~6.1 and Prop.~6.3}
  imply that each residue class mod~$p$ on~$C$ contains at most one rational point,
  equivalently, the reduction map $C(\Q) \to C(\F_p)$ is injective.
  Since one easily checks that the nine points given in Theorem~\ref{T:main}
  are indeed on~$C$ and that $\#C(\F_3) = 9$, the result follows.
\end{proof}

\begin{lemma} \label{L:5}
  The torsion subgroup of~$J(\Q)$ is isomorphic to $\Z/2\Z \times \Z/18\Z$
  and is contained in the group generated by differences of the known rational
  points on~$C$. This latter group has rank~$1$ and is saturated in~$J(\Q)$.
  In particular, if $J(\Q)$ has rank~$1$, then $J(\Q)$ is the group generated
  by differences of the known rational points.
\end{lemma}

\begin{proof}
  Using Magma, we compute that
  \[ J(\F_3) \cong \Z/2\Z \times \Z/252\Z \qquad\text{and}\qquad
     J(\F_5) \cong \Z/2\Z \times \Z/18\Z \times \Z/36\Z \,.
  \]
  Since the rational torsion subgroup of~$J$ maps injectively into~$J(\F_p)$
  when $p$ is odd and of good reduction, this implies that the torsion
  subgroup of~$J(\Q)$ embeds into~$\Z/2\Z \times \Z/36\Z$.

  On the other hand, we determine (again using Magma) that the
  group~$J(\Q)_{\known}$ generated by the differences of the nine known
  rational points on~$C$ is isomorphic to $\Z/2\Z \times \Z/18\Z \times \Z$.
  In particular, we know the full $2$-torsion subgroup of~$J(\Q)$.
  We then check that none of the three nontrivial rational $2$-torsion
  points are divisible by~$2$ in~$J(\Q)$ by verifying that for some
  good odd prime~$p$, the image of the point in~$J(\F_p)$ is not
  divisible by~$2$ in~$J(\F_p)$. This implies that there is no rational
  point of order~$36$ on~$J$. This finishes the proof of the first
  two statements.

  To see that $J(\Q)_{\known}$ is saturated in~$J(\Q)$ (i.e., if
  $P \in J(\Q)$ and $n \ge 1$ are such that $nP \in J(\Q)_{\known}$, then
  already $P \in J(\Q)_{\known}$), we observe that $J(\Q)_{\known}$
  surjects onto~$E(\Q)$ under the homomorphism $\psi \colon J \to E$ induced
  by the double cover $C \to E$. If $P \in J(\Q)$ with $nP \in J(\Q)_{\known}$,
  then there is $P' \in J(\Q)_{\known}$ such that $\psi(P') = \psi(P)$;
  since $nP$ and~$nP'$ are both in~$J(\Q)_{\known}$ and have the same
  image on~$E$, $n(P - P')$ and therefore $P - P'$ must be torsion, hence
  in~$J(\Q)_{\known}$ by the above,
  and so $P = P' + (P - P') \in J(\Q)_{\known}$.
\end{proof}


\section{The conditional analytic proof of Theorem \ref{T:main}} \label{S:analytic}

We now proceed to give a first proof of the fact that $J(\Q)$ has rank~$1$.
By Lemma~\ref{L:4}, this then implies Theorem~\ref{T:main}.

This proof is ``analytic'' in the sense that it uses the $L$-series of~$J$.
Assuming standard conjectures (see, e.g.,~\cite{Hulsbergen}*{Conj.~3.1.1}),
the $L$-series of~$J$ extends to an entire
function and satisfies the functional equation
\[ \Lambda(J, 2-s) = \pm \Lambda(J, s) \quad\text{with}\quad
   \Lambda(J, s) = N^{s/2} (2 \pi)^{-4s} \Gamma(s)^4 L(J, s) \,,
\]
where $N = 2^6 \cdot 23 \cdot 29^2$ is the conductor (Lemma~\ref{L:3}).
We compute sufficiently many coefficients of the $L$-series by counting
points on~$C$ over suitable finite fields; the Euler factors at the bad
primes $2, 23, 29$ can be derived from the information given
in Lemma~\ref{L:1} (and from the zeta function of the desingularized
components of the special fiber) to be $(1 + T)^2$ at~$2$,
$(1 + T) (1 - 6T + 23T^2) (1 + 23T^2) (1 + 8T + 23T^2)$ at~$23$
and $(1 - T^2) (1 - 5T + 28T^2 - 145T^3 + 841T^4)$ at~$29$.

We then check for which sign the functional equation holds numerically
(up to the precision used). This shows that the negative sign is the
correct one. Finally, we evaluate the derivative of $L(J, s)$
at~$s = 1$ numerically and find that it is nonzero. (The code Magma uses
requires that the sign in the functional equation is known.)
Together, this shows
that the order of vanishing of~$L(J, s)$ at $s = 1$ is~$1$
(assuming the sign in the functional equation is~$-1$).
Consequently, we have the following. (See~\cite{Tate} for the BSD Conjecture.)

\begin{proposition}
  If $L(J, s)$ extends to an entire function and satisfies the expected
  type of functional equation, and if the BSD rank conjecture holds for~$J$,
  then $J(\Q)$ has rank~$1$, and Theorem~\ref{T:main} follows.
\end{proposition}

\begin{remark*}
  It is more efficient to perform the computation for the $L$-series
  associated to~$A$, since this requires fewer coefficients to be
  known to obtain the result to a given precision. We check that
  the correct sign of the functional equation for~$L(A,s)$ is~$+1$
  and that $L(A,1)$ is nonzero (note that $L(J,s) = L(E,s) L(A,s)$
  and that the BSD rank conjecture is known for~$E$ by~\cite{Kolyvagin},
  since $E$ has analytic rank~$1$ as can be seen from its LMFDB entry).
\end{remark*}


\section{The unconditional algebraic proof of Theorem \ref{T:main}} \label{S:algebraic}

We now give a second proof that the rank of~$J(\Q)$ is~$1$. This is based
on the exact sequence
\[ 0 \To J(\Q)/2J(\Q) \To \Sel_2(J/\Q) \To \Sha(J/\Q)[2] \To 0 \,, \]
where $\Sel_2(J/\Q)$ is the $2$-Selmer group and $\Sha(J/\Q)$ is
the Tate-Shafarevich group of~$J$. All groups in the sequence are
vector spaces over~$\F_2$. This implies that we obtain an upper bound
on the rank of~$J(\Q)$ that is given by (recall that $\dim_{\F_2} J(\Q)[2] = 2$
by Lemma~\ref{L:5})
\[ \dim_{\F_2} \Sel_2(J/\Q) - \dim_{\F_2} J(\Q)[2] = \dim_{\F_2} \Sel_2(J/\Q) - 2 \,. \]
Since we know from Lemma~\ref{L:5} that the rank is at least~$1$, it suffices
to show that
\[ \dim_{\F_2} \Sel_2(J/\Q) \le 3 \,. \]
We will do this by following the approach described in~\cite{BPS2016}.

We first determine the set of odd theta characteristics on~$C$. They are
represented by divisors~$D$ such that $2D$ is the divisor on~$C$
cut out by a $(1,1)$-form on~$\PP^1 \times \PP^1$ (the divisors of $(1,1)$-forms
form the canonical class on~$C$). To do this, we set up a $(1,1)$-form
with (four) undetermined coefficients and form the resultants with
the equation defining~$C$ with respect to each of the two sets of projective coordinates.
Then a necessary condition is that both resultants are squares, up to
constant factors. This defines a subscheme of~$\PP^3$ that turns out to
be of dimension zero and to have the correct degree~$120$. We can then
split this scheme into its irreducible components over~$\Q$. This leads
to the following result.

\begin{lemma} \label{L:thetachars}
  The Galois action on the set of $120$ odd theta characteristics on~$C$
  has orbits of sizes $1, 1, 1, 3, 3, 3, 12, 24, 24, 24, 24$.
  The three rational odd theta characteristics correspond to the
  $(1,1)$-forms (in affine coordinates)
  \[ 2xy - x - y\,, \qquad xy - 1 \qquad\text{and}\qquad x + y \,. \]
  The Galois orbits of length~$3$ contain points defined over~$\Q(\alpha_3)$,
  the orbit of length~$12$ contains points defined over~$\Q(\alpha_{12})$,
  three of the orbits of length~$24$ contain points defined over~$\Q(\alpha_{24})$,
  and the last orbit of length~$24$ contains points defined over~$\Q(\beta_{24})$,
  where
  \begin{align*}
    \alpha_3^3 - \alpha_3^2 - 2 &= 0\,, \\
    \alpha_{12}^{12} - 6 \alpha_{12}^{11} + 10 \alpha_{12}^{10} + 2 \alpha_{12}^9
      - 20 \alpha_{12}^8 + 18 \alpha_{12}^7 \quad& \\
      {} + 12 \alpha_{12}^6 - 12 \alpha_{12}^5
      + 13 \alpha_{12}^4 + 10 \alpha_{12}^3 - 4 \alpha_{12}^2 + 1 &= 0\,, \\
    \alpha_{24}^{24} + 4 \alpha_{24}^{23} + 6 \alpha_{24}^{22} - 9 \alpha_{24}^{20}
      - 42 \alpha_{24}^{19} - 19 \alpha_{24}^{18} - 60 \alpha_{24}^{17} + 70 \alpha_{24}^{16} \quad & \\
      {} - 46 \alpha_{24}^{15} + 196 \alpha_{24}^{14} - 132 \alpha_{24}^{13}
      + 225 \alpha_{24}^{12} - 290 \alpha_{24}^{11} + 195 \alpha_{24}^{10} - 256 \alpha_{24}^9 \quad & \\
      {} + 177 \alpha_{24}^8 - 82 \alpha_{24}^7 + 76 \alpha_{24}^6 - 8 \alpha_{24}^5
      - 2 \alpha_{24}^4 - 10 \alpha_{24}^2 - 2 &= 0\,, \\
    \beta_{24}^{24} + \beta_{24}^{23} - 13 \beta_{24}^{22} + 9 \beta_{24}^{21} + 78 \beta_{24}^{20}
      - 116 \beta_{24}^{19} - 73 \beta_{24}^{18} + 897 \beta_{24}^{17} - 189 \beta_{24}^{16} \quad & \\
      {} - 1955 \beta_{24}^{15} + 4105 \beta_{24}^{14} + 755 \beta_{24}^{13} - 8996 \beta_{24}^{12}
      + 13926 \beta_{24}^{11} - 8747 \beta_{24}^{10} + 1395 \beta_{24}^9 \quad & \\
      {} + 690 \beta_{24}^8 + 3152 \beta_{24}^7 - 5140 \beta_{24}^6 + 3260 \beta_{24}^5
      - 272 \beta_{24}^4 - 584 \beta_{24}^3 + 308 \beta_{24}^2 - 68 \beta_{24} + 8 &= 0\,.
  \end{align*}
  All the odd theta characteristics are defined over the splitting field~$K$
  of~$\Q(\alpha_{24})$. The Galois group~$G$ of~$K$ over~$\Q$ has order~$1152$;
  there is, up to conjugation, exactly one subgroup of~$\Sp_8(\F_2)$ that is
  isomorphic to it and whose natural action on~$\F_2^8$ has orbits of the lengths
  given above (plus further orbits).
\end{lemma}

We will from now on identify an odd theta characteristic with the unique effective
divisor in its linear equivalence class.

We can use one of the rational odd theta characteristics, $D_0$, say, as a basepoint;
then for each odd theta characteristic~$D$, $[D - D_0]$ is an element of~$J[2]$.
We check that these differences, where $D$ runs through the theta characteristics
in the other two orbits of length~$1$, the orbit of length~$12$ and the first three orbits
of length~$24$, generate~$J[2]$. We can therefore use this set of theta characteristics
to define a ``true descent setup'' in the terminology of~\cite{BPS2016} for
a $2$-descent on~$J$.

Let $\varphi_0$ be a $(1,1)$-form whose divisor is twice~$D_0$ and pick
further $(1,1)$-forms $\varphi_1, \ldots, \varphi_6$ whose divisors are
twice $D_1, \ldots, D_6$, where $D_1$ and~$D_2$ are the other two rational
odd theta characteristics, $D_3$ is a representative of the Galois orbit
of size~$12$, and $D_4, \ldots, D_6$ are representatives of the three Galois
orbits of size~$24$ corresponding to~$\Q(\alpha_{24})$. Then evaluating
the tuple of rational functions $(\varphi_1/\varphi_0, \ldots, \varphi_6/\varphi_0)$
on divisors of degree zero (with support disjoint from that of $D_0, D_1, \ldots, D_6$)
gives, for each field extension~$K$ of~$\Q$, a map
\[ J(K) \stackrel{\gamma_K}{\To}
          \Bigl(\frac{K^\times}{K^{\times 2}}\Bigr)^2
          \times \frac{(K \otimes \Q(\alpha_{12}))^\times}{(K \otimes \Q(\alpha_{12}))^{\times 2}}
          \times \Bigl(\frac{(K \otimes \Q(\alpha_{24}))^\times}{(K \otimes \Q(\alpha_{24}))^{\times 2}}\Bigr)^3
        =: H_K \,,
\]
which is compatible with homomorphisms of fields and factors as a composition
\[ J(K) \stackrel{\delta_K}{\To} H^1(K, J[2]) \stackrel{w_K}{\To} H_K \,, \]
where $\delta_K$ is the connecting map in Galois cohomology coming from the Kummer sequence
\[ 0 \To J[2] \To J \To J \To 0 \]
and $w_K$ is induced by the Weil pairing with $[D_1 - D_0], \ldots, [D_6 - D_0] \in J[2]$.
We omit the subscript~$K$ when $K = \Q$ and use the subscript~$v$ in place of~$\Q_v$.

\begin{lemma} \label{L:w_injective}
  The map $w = w_{\Q}$ is injective.
\end{lemma}

\begin{proof}
  We can verify that $w$ is injective with a finite computation. As usual,
  $\mu_2$ denotes the Galois module of second roots of unity, and for a
  number field~$K$, $R_{K/\Q} \mu_2$ denotes the Weil restriction of scalars
  of~$\mu_2$ over~$K$ to~$\Q$. (Concretely, as a group, $R_{K/\Q} \mu_2$ is
  $\mu_2(K \otimes_{\Q} \bar{\Q}) \simeq \mu_2^{[K : \Q]}$, with factors
  corresponding to the various embeddings $K \to \bar{\Q}$; the Galois action
  permutes the factors in the same way as it permutes the embeddings.
  This is the same as the ``twisted power''~$\mu_2^\Delta$ in the notation
  of~\cite{BPS2016}, where $\Delta$ is the Galois orbit of a theta characteristic
  defined over~$K$.) Then the long exact sequence in Galois cohomology induced by
  \[ 0 \To J[2] \To \underbrace{\mu_2^2 \times R_{\Q(\alpha_{12})/\Q} \mu_2
                                  \times \bigl(R_{\Q(\alpha_{24})/\Q} \mu_2\bigr)^3}_{=: W}
                \stackrel{q}{\To} Q \To 0 \,,
  \]
  where the first map comes from the Weil pairing and induces~$w$ on~$H^1$
  (in particular, $H = H^1(\Q, W)$ by the Kummer sequence, Hilbert's Theorem~90
  and Shapiro's Lemma)
  and $Q$ is defined to be its cokernel, shows that $\ker(w) \simeq Q(\Q)/q(W(\Q))$.
  The Galois modules $W$ and~$Q$ are finite, and so we can verify that
  $Q(\Q) = q(W(\Q))$; this shows that $w$ has trivial kernel.
\end{proof}

From the diagram
\[ \xymatrix{ J(\Q) \ar[r]^-{\delta} \ar[d] & H^1(\Q, J[2]) \ar[r]^-{w} \ar[d] & H \ar[d]^-{\rho} \\
              \displaystyle\prod_v J(\Q_v) \ar[r]^-{\prod_v \delta_v}
               & \displaystyle\prod_v H^1(\Q_v, J[2]) \ar[r]^-{\prod_v w_v}
               & \displaystyle\prod_v H_v
            }
\]
(where $v$ runs through the places of~$\Q$), we obtain a map
\[ \Sel_2(J/\Q) \To \bigl\{\xi \in H : \rho(\xi) \in \prod_v \im(\gamma_v)\bigr\} =: S \,, \]
and since $w$ is injective, this is an embedding. We now show the following.

\begin{proposition} \label{P:Selmer}
  We have that $\dim_{\F_2} S = 3$.
\end{proposition}

\begin{proof}
  By general theory (see~\cite{BPS2016}*{Prop.~9.2}), we know that $S$ is contained in the
  subgroup of~$H$ consisting of elements that are unramified (i.e., such that
  adjoining a square root gives an unramified extension) outside~$2$ and the
  odd primes~$p$ such that the Tamagawa number of~$J$ at~$p$ is even.
  Since by Lemma~\ref{L:1}, the special fiber of the minimal regular model
  of~$C$ at each of the two odd bad primes has only one component, the Tamagawa
  number is~$1$ for all odd primes. We conclude that $S$ is contained in the
  subgroup of elements unramified outside~$2$.

  A short computation with, e.g., Magma shows that $\Q(\alpha_{12})$ has trivial
  class group. The Minkowski bound for~$\Q(\alpha_{24})$ is sufficiently small
  to allow for an unconditional determination of the class group; Magma takes
  about 5~hours to show that this group is also trivial. Combined with the
  discussion in the preceding paragraph, this tells us that $S$ is contained
  in the product~$S_0$ of the groups of $\{2\}$-units modulo squares of the fields corresponding
  to the various factors, namely $\Q$, $\Q$, $\Q(\alpha_{12})$,
  $\Q(\alpha_{24})$, $\Q(\alpha_{24})$, $\Q(\alpha_{24})$.
  We can then compute this (finite) group~$S_0$, which gives us a first upper
  bound on~$S$. We also find the image of the group generated by differences
  of known rational points in~$S$ under~$\gamma$ and check that it has
  dimension~$3$.

  The next step is to find $\im(\gamma_2) \subset H_2$. Its dimension is at
  most the dimension of~$J(\Q_2)/2J(\Q_2)$, which, by general theory
  (see~\cite{Schaefer1998}*{Prop.~2.4}), is $\dim J(\Q_2)[2] + 4$
  (where $4$ is the genus).
  We can use the Galois action on the theta characteristics to determine
  that $\dim J(\Q_2)[2] = 4$. We then use the known rational points (taking
  one of them as a basepoint) and further points that are $2$-adically close
  to them (obtained by parameterizing the residue disk by power series) to
  produce a number of elements of the image. We then check that these elements
  generate a subspace of dimension~$8$, which implies that this subspace is
  the image of~$\gamma_2$.

  Finally, we set up the reduction map
  $\rho_2 \colon S_0 \hookrightarrow H \to H_2$ and compute
  $\rho_2^{-1}\bigl(\im(\gamma_2)\bigr) \subset S_0$.
  This subgroup of~$S_0$ contains~$S$, and we verify that it coincides
  with the image under~$\gamma$ of the subgroup of~$J(\Q)$ generated
  by the known points. This implies that $S = \gamma(J(\Q))$; in particular,
  this group has dimension~$3$.
\end{proof}

We can then deduce Theorem~\ref{T:main} without relying on the BSD conjecture.

\begin{corollary}
  The Mordell-Weil group $J(\Q)$ has rank~$1$, and $C(\Q)$ consists
  of the nine known rational points.
\end{corollary}

\begin{proof}
  By Lemma~\ref{L:5} and Proposition~\ref{P:Selmer}, we have that
  \[ 1 \le \rk J(\Q) \le \dim_{\F_2} \Sel_2(J/\Q) - 2 \le \dim_{\F_2} S - 2 = 1 \,; \]
  this shows the first statement. The second statement then follows
  from Lemma~\ref{L:4}.
\end{proof}


\section{The BSD conjecture for~$J$} \label{S:BSD}

We can use the information obtained from the two approaches we have presented
to provide a partial numerical verification of the strong form of the
Birch and Swinnerton-Dyer conjecture for the Jacobian~$J$ of our curve.
Note that what we have done already proves the weak form, assuming that
the $L$-series of~$J$ has an analytic continuation and satisfies the expected
functional equation (which, assuming the former, it does numerically).
Recall that the conjecture predicts in our case that
\[ L'(J, 1) = \frac{\prod_p c_p \cdot \Omega_{J/\Q} \cdot R_{J/\Q} \cdot \#\Sha(J/\Q)}%
                   {\#J(\Q)_{\tors}^2} \,,
\]
where $p$ runs through the bad primes, $c_p$ is the Tamagawa number of~$J$ at~$p$,
$\Omega_{J/\Q}$ is the ``real period'', i.e., the integral of a top N\'eron
differential form on~$J$ over the real points~$J(\R)$, $R_{J/\Q}$ is the regulator
of~$J(\Q)$, which in our rank~$1$ situation is simply the canonical height
$\hat{h}(P)$ of a generator~$P$ of the free part of~$J(\Q)$, and $\Sha(J/\Q)$
is the Tate-Shafarevich group of~$J$, which is conjectured (but not known)
to be finite. The left hand side of the equation has been computed in the
course of our analytic proof.

We have already seen that $c_{23} = c_{29} = 1$. We can compute the component
group of the N\'eron model of~$J$ at~$2$ from the configuration of the
special fiber given in Lemma~\ref{L:1} and find that the group consists of
elements fixed by Frobenius and has structure $\Z/2\Z \times \Z/6\Z$,
so that $c_2 = 12$. Lemma~\ref{L:5} shows that $\#J(\Q)_{\tors} = 36$.

We determine~$\Omega_{J/\Q}$ as follows. Since the given model is already
regular at all primes, a basis of the N\'eron $1$-forms on~$C$ is given
by the standard differentials
\[ \frac{x^j y^k}{\frac{\partial F}{\partial y}(x,y)}\,dx
    = -\frac{x^j y^k}{\frac{\partial F}{\partial x}(x,y)}\,dy \,, \qquad
   (j,k) \in \{(0,0), (0,1), (1,0), (1,1)\} \,,
\]
where $F(x,y) = 0$ is the affine equation defining the curve. We integrate
this basis of differentials on a basis of the integral homology~$H_1(C(\C), \Z)$
to obtain the ``big period matrix'' of~$J$. Magma has functionality
for computing the big period matrix for the Jacobian of a plane curve.
The real period is the covolume of the real lattice generated
by the traces (twice the real parts) of the columns of the big period matrix.

To compute the regulator, we recall from Lemma~\ref{L:5} (and using
that we now know that $J(\Q)$ has rank~$1$) that we know the full group~$J(\Q)$.
We can compute the regulator by relating the height of a generator~$P$
of the free part of~$J(\Q)$ to the height of its image~$P_E$ on the elliptic
curve~$E$, which is a generator of~$E(\Q)$, as follows.
Let $\phi \colon E \to J$ be the morphism induced via Picard functoriality
by the double cover $C \to E$. Then $\phi(P_E) = 2 P$. The compatibility
of heights with morphisms implies that $m \hat{h}_E(P_E) = \hat{h}_J(\phi(P_E))$,
where $m = \deg \phi^* \Theta_J$ is the intersection number of the image
of~$E$ in~$J$ with the theta divisor of~$J$ (note that both canonical
heights are associated to Weil heights coming from twice the corresponding
theta divisor). This gives
\[ R_{J/\Q} = \hat{h}_J(P)
            = \tfrac{1}{4} \hat{h}_J(2P)
            = \tfrac{1}{4} \hat{h}_J(\phi(P_E))
            = \tfrac{m}{4} \hat{h}_E(P_E)
            = \tfrac{m}{4} R_{E/\Q} \,.
\]
By~\cite{ACGH}*{Page 370, Exercise~VIII.D-10}, we find that $m = 2$ (note that
$n = 2$ is the degree of $C \to E$ and $h = 1$ is the genus of~$E$
in the notation there), so that
$R_{J/\Q} = \tfrac{1}{2} R_{E/\Q}$.

We finally solve for the order of~$\Sha(J/\Q)$. The result is that
\[ \#\Sha(J/\Q) \approx 1 \]
to the precision of our computation.
This is consistent with what we know about $\Sha(J/\Q)$, namely that it
is supposed to be finite, that its order is a square (if finite) because
of the existence of the Cassels-Tate pairing and the fact that $C$ has
rational points (see~\cite{PoonenStoll1999b}) and that its order must be
odd, since $\dim_{\F_2} \Sel_2(J/\Q) = \dim_{\F_2} J(\Q)[2] + \rk J(\Q)$.

\begin{remark*}
  An alternative way to determine~$m$, and indeed the approach we originally
  followed, is to use the method described in~\cite{vBHM} to obtain
  a numerical approximation of~$\hat{h}_J(P)$ by computing the
  height pairing between two divisors that are differences of distinct
  rational points on~$C$ and whose images in~$J(\Q)$ are known multiples
  of~$P$ up to adding torsion points. (We note that the code that was generously
  supplied to us by Steffen M\"uller had to be tweaked
  a bit to work for this example.) This shows that
  $m/4 = \hat{h}_J(P)/\hat{h}_E(P_E)$ is close to~$1/2$, and since
  we know that $m$ is a natural number, this proves that $m = 2$.
\end{remark*}


\end{document}